\documentclass[reqno, 12pt, reqno]{amsart}

\usepackage{amsfonts, amsthm, amsmath, amssymb}
\usepackage{hyperref}
\hypersetup{colorlinks=false}

\usepackage[margin=1.5in]{geometry}

\usepackage[OT2,T1]{fontenc}
\DeclareSymbolFont{cyrletters}{OT2}{wncyr}{m}{n}
\DeclareMathSymbol{\Sha}{\mathalpha}{cyrletters}{"58}
\usepackage{helvet}

\RequirePackage{mathrsfs} \let\mathcal\mathscr

\numberwithin{equation}{section}

\newtheorem{theorem}{Theorem}[section]
\newtheorem{lemma}[theorem]{Lemma}

\theoremstyle{definition}
\newtheorem*{ack}{Acknowledgements}
\newtheorem{remark}[theorem]{Remark}

\renewcommand{\phi}{\varphi}
\renewcommand{\rho}{\varrho}

\newcommand{\PP}{\mathbb{P}}

\newcommand{\A}{\mathbf{A}}
\newcommand{\FF}{\mathbb{F}}
\newcommand{\ZZ}{\mathbb{Z}}

\newcommand{\QQ}{\mathbb{Q}}
\newcommand{\RR}{\mathbb{R}}

\renewcommand{\H}{\mathrm{H}}

\newcommand{\cD}{\mathcal{D}}

\newcommand{\cF}{\mathcal{F}}

\newcommand{\Gal}{{\rm Gal}}

\renewcommand{\leq}{\leqslant}

\renewcommand{\geq}{\geqslant}

\renewcommand{\bar}{\overline}

\newcommand{\ve}{\varepsilon}

\DeclareMathOperator{\rank}{rank}
\DeclareMathOperator{\disc}{disc}

\DeclareMathOperator{\Pic}{Pic}

\DeclareMathOperator{\meas}{meas}

\DeclareMathOperator{\Spec}{Spec}

\DeclareMathOperator{\sel}{Sel}

\DeclareMathOperator{\Br}{Br}

\newcommand{\GL}{\mathrm{GL}_2}
\newcommand{\SL}{\mathrm{SL}_2}

\renewcommand{\hat}{\widehat}

\begin{document}

\date{\today}

\title{Many cubic surfaces contain rational points}

\author{T.D. Browning}
\address{School of Mathematics\\
University of Bristol\\ Bristol\\ BS8 1TW
\\ UK}
\email{t.d.browning@bristol.ac.uk}

\dedicatory{For Klaus Roth, in memoriam.}

\begin{abstract}
Building on recent work of Bhargava--Elkies--Schnidman  and Kriz--Li,
we   produce infinitely many smooth cubic surfaces defined  over the field of rational numbers that contain rational points.
\end{abstract}

\date{\today}


\maketitle

\thispagestyle{empty}

\section{Introduction}

This paper is concerned with the Hasse principle for the family of projective cubic surfaces
\begin{equation}\label{eq:cubics}
f(x_0,x_1)=g(x_2,x_3),
\end{equation}
where $f,g\in \QQ[u,v]$ are binary cubic forms.  Our main result shows that a positive proportion of these surfaces, when ordered by height, possess a $\QQ$-rational point.  First, we discuss the question of local solubility.
The following result
will be addressed in
  \S \ref{s:local-er}.
\begin{theorem}\label{t:local}
Approximately  99\% of the cubic surfaces \eqref{eq:cubics}, when ordered by height, 
are everywhere locally soluble.
\end{theorem}

It has long been known that the Hasse principle does not always hold for cubic surfaces and so local solubility  is not enough to ensure that the surface \eqref{eq:cubics} has a $\QQ$-rational point.  In the special case that  $f$ and $g$ are diagonal  there are many known counter-examples to the Hasse principle, the most famous being the surface
$$
5x_0^3+9x_1^3+10x_2^3+12x_3^3=0,
$$
that was discovered by Cassels and Guy \cite{CG}.
As explained by 
Manin \cite{manin}, this example is accounted for by the 
Brauer--Manin obstruction.
It has been conjectured by Colliot-Th\'el\`ene and Sansuc \cite[\S V]{ct}
 that  this obstruction explains all failures of the Hasse principe for  smooth cubic surfaces.   This conjecture remains wide open in general. However, work of Swinnerton-Dyer \cite{aens}  confirms it for a special  family of  diagonal cubic surfaces, conditionally under the assumption that the Tate--Shafarevich group of elliptic curves is finite.

A recent investigation of Bright \cite{hp} has focused on families of varieties over number fields that have no Brauer--Manin obstruction to the Hasse principle.
In \S \ref{s:bright} we shall check that the conditions of his main result are satisfied for the family of cubic surfaces \eqref{eq:cubics}, thereby leading to the following conclusion.

\begin{theorem}[Bright]\label{t:bright}
Assume  that the Brauer--Manin obstruction is the only obstruction to the Hasse principle for smooth cubic surfaces. Then 100\% of the cubic surfaces \eqref{eq:cubics}, when ordered by height,  satisfy the Hasse principle.
\end{theorem}

There are very few results which establish the existence of $\QQ$-rational points on cubic surfaces unconditionally.
Given coprime $a,b\in \ZZ$, 
Heath-Brown and Moroz 
\cite[Cor.~1.2]{HB-M}
have shown that  any 
cubic surface
$$
x_0^3+2x_1^3+ax_2^3+bx_3^3=0
$$
has a $\QQ$-rational point, provided that   either 
$a  \equiv \{\pm 2, \pm 3\} \bmod{9}$, or 
$b  \equiv \{\pm 2, \pm 3\} \bmod{9}$, or 
$
a\equiv \pm b \bmod{9}.
$
This result combines sophisticated sieve methods, which allow primes to be represented by binary cubic polynomials, 
with a result of  Satg\'e \cite[Prop.~3.3]{satge},  which demonstrates that the
curve
$
u^3+2v^3=pz^3
$
has a $\QQ$-rational point for any prime  $p\equiv 2\bmod{9}$.

Our main result provides a much denser set of smooth cubic surfaces
which are in possession of a rational point.

\begin{theorem}\label{t:main}
A positive proportion of the cubic surfaces \eqref{eq:cubics}, when ordered by height,  possess a $\QQ$-rational point.
\end{theorem}

Recall that 
the set of $\QQ$-rational points is 
Zariski dense  on 
the cubic surface
\eqref{eq:cubics} as soon as it is non-empty,
as proved by
 Segre \cite{segre}.

Theorem \ref{t:main} will be established in 
\S \ref{s:global}.
Our  proof 
follows a strategy of Bhargava  
\cite[\S 4]{manjul}, which was developed 
to show that a positive proportion of  plane cubic curves have a $\QQ$-rational point.
The key input is recent work of   Kriz and Li \cite{kriz}, concerning  the family of Mordell curves
$$
E_k:\quad y^2=x^3+k,
$$
for $k\in \ZZ$.  Using a Heegner point construction, Kriz and Li show that a positive proportion of these elliptic curves have  rank 1.
The curves $E_k$ have also been the focus of 
novel work by Bhargava, Elkies and Schnidman \cite{BES} and we shall  draw heavily on their investigation.  It follows from \cite[Thm.~8]{BES}, in particular, that 
at least 41.1\% of the curves $E_k$, for $k\in \ZZ$, have rank 1, 
but this is only achieved under the assumption that the Tate--Shaferevich group is finite, whereas the work of Kriz and Li is unconditional.

Associated to $E_k$ is an obvious $3$-isogeny $\phi$
and we shall monopolise on the correspondence   between  elements of the $\phi$-Selmer group  and locally soluble   binary cubic forms. 
Let $V(\ZZ)$ be the lattice of (integer-matrix) binary cubic forms 
$$
f(u,v)=au^3+3bu^2v+3cuv^2+dv^3,
$$
for $a,b,c,d\in \ZZ$. Associated to each $f\in V(\ZZ)$ is the genus $1$
plane curve
\begin{equation}\label{eq:Cf}
C_f=\{z^3=f(u,v)\}\subset \PP^2.
\end{equation}
In order to establish Theorem \ref{t:main},  it will suffice to show that there is a positive proportion of irreducible $f\in V(\ZZ)$, when ordered by height, such that  $C_f(\QQ)\neq \emptyset$.  To see this we first note that, since  $f$ is irreducible,  
we must have $z\neq 0$ in any rational point
$(u:v:z)\in C_f(\QQ)$.  Now if the binary cubic forms $f,g\in V(\ZZ)$ give rise to the rational points 
$(u:v:z)\in C_f(\QQ)$ and $(u':v':z')\in C_g(\QQ)$, with $zz'\neq 0$, then it follows that the 
rational point $(z'u:z'v:zu':zv')\in \PP^3(\QQ)$ lies on the cubic surface \eqref{eq:cubics}.

\begin{remark}
It is also possible to say something about appropriate K3 surfaces. Consider the 
family of Kummer surfaces with affine equation
\begin{equation}\label{eq:kummer}
z^2=g_1(x)g_2(y),
\end{equation}
where $g_1,g_2$ are irreducible quartic polynomials defined over  $\QQ$.
Recent work of Harpaz and Skorobogatov \cite{harpaz} establishes the Hasse principle for these surfaces, provided that $g_1,g_2$ satisfy suitable genericity conditions, albeit under an assumption about the finiteness of 
certain associated  Tate--Shaferevich groups.
It has been shown by Bhargava \cite[Thm.~11]{manjul} that a positive proportion of the genus 1 curves
$
y^2=g(u,v)
$
possess a $\QQ$-rational point, 
when the binary quartic forms $g\in \QQ[u,v]$
are ordered by height. An immediate consequence of this
is the unconditional statement that 
a positive proportion of the Kummer surfaces \eqref{eq:kummer}, when ordered by height,  possess a $\QQ$-rational point. 
\end{remark}

\begin{ack}
The author is very grateful to Tim Dokchitser,  Adam Morgan,  Jack Thorne
and Olivier Wittenberg for useful conversations, and to the anonymous referee for several helpful comments.
Thanks are also due to J\"org Jahnel for help with Remark \ref{rem:JJ}.
While working on this paper the   author was
supported by ERC grant \texttt{306457}.
\end{ack}

\section{The family of cubic surfaces and local solubility}\label{s:local}	

Let  $X\subset \PP^7\times \PP^3$ be the biprojective hypersurface defined by 
$$
ax_0^3+
bx_0^2x_1+cx_0x_1^2+dx_1^3=ex_2^3+fx_2^2x_3+gx_2x_3^2+hx_3^3.
$$
We obtain a flat surjective morphism $\pi: X\to \PP^7$ by projecting to the point $(a:b:\dots:h)$, with $X$ smooth, projective and geometrically integral.  
We denote by $X_P$ the fibre above any point $P=(a:b:\dots:h)\in \PP^7(\QQ)$.
Let $\eta: \Spec(K)\to \PP^7$ denote the 
generic point and $\bar \eta: \Spec(\bar K) \to \PP^7$
a geometric point
above $\eta$, where $K$ is the function field of $\PP^7$ and $\bar K$ is an algebraic closure of $K$.
The geometric generic fibre $X_{\bar \eta}$ is a smooth cubic surface and so it is  connected and has torsion-free
Picard group. Denote by $\bar X$ the base change of $X$ to an algebraic closure $\bar \QQ$ of
$\QQ$ and let $\A_\QQ$ denote the ring of ad\`eles of $\QQ$. 

\subsection{Proof of Theorem \ref{t:bright}}\label{s:bright}

The statement  follows from  \cite[Thm.~1.4]{hp}, provided that the following hypotheses are met:
\begin{enumerate}
\item \label{1}
$X(\A_\QQ) \neq  \emptyset$;
\item \label{2}
the fibre of $\pi$ at each codimension-1 point of $\PP^7$
is geometrically integral;
\item \label{3}
the fibre of $\pi$ at each codimension-2 point of $\PP^7$
has a geometrically reduced
component;
\item \label{4}
$\H^1(\QQ,\Pic \bar X) = 0$;
\item \label{5}
$\Br \bar X $ is trival;
\item\label{6}
$\H^2(\QQ,\Pic \PP^7)\to \H^2(\QQ, \Pic \bar X)$
is injective;
\item \label{7}
$\Br X_{\bar \eta}=0$.
\end{enumerate}
Since $X$ is $\QQ$-rational the conditions (\ref{1}), (\ref{4}) and (\ref{5}) are automatically met. 
The cubic surface $X_P$ has equation $f(x_0,x_1)=g(x_2,x_3)$ for suitable binary cubic forms $f $ and $g$. If $X_P$ fails to be be geometrically integral then it must be singular, which can only happen when $\disc(f)=\disc(g)=0$. This is a codimension-2 condition, which thereby establishes (\ref{2}). The only way that (\ref{3}) can fail is if $f$ and $g$ are proportional and both have repeated roots. But this is a codimension-3 condition on the set of  coefficients.
Condition (\ref{6}) follows from \cite[Prop.~5.17]{wa}.
Finally, $X_{\bar \eta}$ is a smooth cubic surface over an algebraically closed field of characteristic $0$. Thus it is rational and it follows that 
$\Br X_{\bar \eta}=0$, as required for  condition (\ref{7}).
This completes the proof of Theorem \ref{t:bright}.

\begin{remark}\label{rem:JJ}
Although we shall not need it in our work, it is possible to show that 
$\Br X_\eta/\Br K=0$ for the family of cubic surfaces considered here, 
using an argument suggested to us by J\"org Jahnel.
According to  Swinnerton-Dyer \cite{SD93},
there is no Brauer--Manin obstruction whenever the action of $\Gal(\bar \QQ/\QQ)$ on the $27$ lines is as large as possible. 
A binary cubic form splits into 3 linear forms, generically by an   $S_3$-extension. Thus, generically, there is an $S_3\times S_3$-extension $k/\QQ$, such that the surface may be written in the classical Cayley--Salmon form
$l_1l_2l_3 = l_4l_5l_6 $, for binary linear forms $l_1,l_2,l_3\in k[x_0,x_1]$ and 
$l_4,l_5,l_6\in k[x_2,x_3]$.
There are 9 obvious $k$-rational lines, given by $l_i=l_j=0$ for $i\in \{1,2,3\}$ and $j\in \{4,5,6\}$,  and 6 obvious tritangent $k$-rational planes. The latter form a pair of Steiner trihedra, both of which are defined over $\QQ$. The group $S_3 \times  S_3 \times S_3$ is the maximal Galois group stabilising two such Steiner trihedra. 
Consider now the cubic surface $S\subset \PP^3$ with equation
$$
3x_0^3 + 7x_0^2x_1 + 11x_0x_1^2 + 13x_1^3=
17x_2^3 + 19x_2^2x_3 + 23x_2x_3^2 + 31x_3^3.
$$
A calculation in \texttt{magma} shows that the Galois group operating on the 27 lines is 
precisely  the maximal possible group $S_3 \times S_3 \times S_3$.
%with orbit type $[9,18]$ on the 27 lines.
%This shows that this choice of coefficients is general enough.
Since the Galois group has order 216, we may consult the appendix in work of Jahnel \cite{jahnel} to conclude that $\Br S/\Br\QQ=0.$ This suffices to show the claim.

\end{remark}

\subsection{Proof of Theorem \ref{t:local}: initial steps} 
\label{s:local-er}
Conditions 
(\ref{1}) and (\ref{2}) in \S \ref{s:bright} are sufficient to ensure that all of the hypotheses of \cite[Thm.~1.3]{wa} are met. Thus  it follows that 
$$
\lim_{H\to \infty}\frac{\#\{(a,b,\dots,h)\in (\ZZ\cap [-H,H])^8: X_{(a:b:\dots:h)}(\A_\QQ)\neq \emptyset  \}}{
(2H)^8}
=
\prod_v \vartheta_v>0,
$$
where
$\vartheta_v$ is the density of cubic surfaces $X_P$ which have $\QQ_v$-points, for each place $v$ of $\QQ$.  
Clearly $\vartheta_\infty=1$. We shall show that 
$
\vartheta_p=1-\chi_p(1/p)
$
for a suitable rational function $\chi_p\in \QQ(x)$.
Let 
$$
Q(x)=
9(1+x+x^2+x^3+x^4)^2(1+x+x^2)(1-x+x^2)(1-x^2)
$$
Then we shall show that 
$$
\chi_p(x)=\frac{(1+x^4)(1+x^2)x^4}{Q(x)}
$$
when $p\equiv 1 \bmod{3}$, while if $p\not \equiv 1 \bmod{3}$ then
$$
\chi_p(x)
\hspace{-0.1cm}
=
\hspace{-0.1cm}
\frac{P(x)x^4
}{
(3-2x^{10})(3-2x^5 )
(1-x+x^2-x^3+x^4)
(1-x^2+x^4)(1+x^2)^3
Q(x)},
$$
where 
\begin{align*}
P(x)=~&
4x^{35} - 4x^{34} + 16x^{33} - 16x^{32} + 56x^{31} - 58x^{30} + 126x^{29} - 132x^{28} \\
&+ 228x^{27} 
- 236x^{26} + 292x^{25} - 314x^{24} + 342x^{23} - 378x^{22}
+ 342x^{21}\\& - 361x^{20} + 197x^{19} - 184x^{18} - 68x^{17} + 110x^{16} - 293x^{15} + 336x^{14} \\ &- 489x^{13}
+ 543x^{12} - 669x^{11} + 702x^{10} - 639x^9 + 639x^8 - 477x^7 \\&+ 450x^6 
- 315x^5 + 306x^4 - 180x^3 + 180x^2 - 45x + 45.
\end{align*}
We note that
$1-\chi_3(1/3)=
0.9965901\dots.
$ 
Inserting our expressions for $\chi_p(1/p)$ into a computer
and taking a product over primes $\leq 10,000$ we find that 
$$
\prod_{\substack{p\leq 10^4\\ p\equiv 1 \bmod{3}}} \vartheta_p=
0.9973776\dots
\quad \text{ and } 
\quad
\prod_{\substack{p\leq 10^4\\ p\equiv 2 \bmod{3}}} \vartheta_p=
0.9998049\dots.
$$
Combining these, one verifies  that  $\prod_p \vartheta_p$ has numerical value
$0.993782\dots$, which is satisfactory for  Theorem \ref{t:local}.

Our calculation of  $\vartheta_p$, for a given prime $p$, is based on work of Bhargava, Cremona and Fisher \cite{local}, who perform the same sort of  calculation for the family of all ternary cubic forms.  
We shall say that a binary cubic form over $\FF_p$ is diagonal if it takes the shape $au^3+bv^3$ for $a,b\in \FF_p^*$.
We  begin by recording the following result.
\begin{lemma}\label{lem:manjul}
\begin{enumerate}
\item
The probability that a random monic cubic over $\FF_p$ has a simple root in $\FF_p$ is 
$\sigma_1=\frac{2}{3}(p^2-1)/p^2$, and the probability that it has a triple root is $\tau_1=1/p^2$.
\item
The probability that a random primitive binary cubic form over $\FF_p$ has a simple root in $\PP^1(\FF_p)$ is 
$\sigma_*=\frac{1}{3}p(2p+1)/(p^2+1)$, and the probability that it has a triple root is $\tau_*=1/(p^2+1)$. 
\item
The probability that a random diagonal binary cubic form over $\FF_p$ has a simple root in $\PP^1(\FF_p)$ is 
$$
\sigma_2
=\begin{cases}
\frac{1}{3} & \text{ if $p\not \equiv 2\bmod{3}$,}\\
1 & \text{ if $p\equiv 2\bmod{3}$.}
\end{cases}
$$

\end{enumerate}
\end{lemma}

\begin{proof}
Part (1) is proved in \cite[Cor.~4]{local}. 
According to  
\cite[Cor.~6]{local},
the probability that a random binary cubic form over $\FF_p$ has a simple root in $\PP^1(\FF_p)$ is 
$\sigma=\frac{1}{3}(p^2-1)(2p+1)/p^3$, and the probability that it has a triple root is $\tau=(p^2-1)/p^4$. 
Part (2) therefore follows on noting that  $\sigma_*=\sigma p^4/(p^4-1)$ and 
$\tau_*=\tau p^4/(p^4-1)$. 
Consider now the special case of diagonal binary cubic forms over $\FF_p$.
The lemma follows on noting that a random element of $\ZZ_p^*$ is a cube in $\QQ_p^*$ with probability $\sigma_2$.
\end{proof}

We now proceed with our calculation of the density $\vartheta_p$ for a given prime $p$.
Over $\QQ_p$, 
cubic surfaces of the shape \eqref{eq:cubics} are determined by pairs $(f,g)$ of binary cubic forms  $f\in \ZZ_p[u,v]$ and $g\in \ZZ_p[x,y]$, where
\begin{align*}
f(u,v)&=c_0u^3+c_1u^2v+c_2uv^2+c_3v^3,\\
g(x,y)&=d_0x^3+d_1x^2y+d_2xy^2+d_3y^3.
\end{align*}
We say that the surface $(f,g)$ is soluble over $\QQ_p$ if there exist $u,v,x,y\in \ZZ_p$, with 
$\min\{v_p(u),v_p(v),v_p(x),v_p(y)\}=0$,
such that $f(u,v)=g(x,y)$. We say that a binary cubic form over $\ZZ_p$ is primitive if not all of its coefficients are divisible by $p$. 
We are only interested in surfaces $(f,g)$ defined over $\ZZ_p$ for which $f$ or $g$ is primitive.

Mimicking  \cite{local}, we shall say that the surface $(f,g)$ has valuations
\begin{align*}
&\geq  \phi_0 ~\geq \phi_1 ~\geq \phi_2 ~\geq \phi_3\\
&\geq \gamma_0 ~\geq \gamma_1 ~\geq \gamma_2 ~\geq \gamma_3,
\end{align*}
if we happen to know that $v_p(c_i)\geq \phi_i$ and $v_p(d_i)\geq \gamma_i$, for $0\leq i\leq 3$. Likewise, we shall replace the inequality symbol with  equality  if and only if we  know the exact $p$-adic valuation of the relevant coefficient. 
We may clearly assume without loss of generality that  $
\min_{0\leq i\leq 3} v_p(c_i)=0$ (i.e. $f$ is primitive).
Through a change of variables, furthermore,  we will also be able to restrict attention to the case where
$\min_{0\leq i\leq 3} v_p(d_i)
=\xi\in \{0,1,2\}$, in which case we write $g=p^{\xi}\tilde g$ for a primitive form $\tilde g$.

We shall examine the solubility of $(f,g)=(f,p^{\xi}\tilde g)$ by considering its reduction modulo $p$. 
Let  $\bar f$ (respectively, $\bar g$) be the reduction of $f$ (respectively, $\tilde g$) modulo $p$. 
Whenever $\bar f$ 
has a simple root over $\FF_p$, then this root can be lifted to a root over $\ZZ_p$ and it follows that the cubic surface $(f,g)$ is soluble over $\QQ_p$.  
We claim the same is true when 
$\bar g$ has a simple root. This is obvious when $\xi=0$ and so we suppose that   $\xi>0$.  In this case 
we merely substitute $pu,pv$ for $u,v$ and divide by $p^{\xi}$ to  obtain 
a new cubic surface $(p^{3-\xi}f,\tilde g)$, which is plainly soluble over $\QQ_p$.

Suppose now that $\bar f$ does not have any  $\FF_p$-roots and that $\bar g$ doesn't have a simple root. Then, either $\xi=0$, in which case $(\bar f,\bar g)$ is absolutely irreducible and it follows that the surface has a smooth $\FF_p$-point which can be lifted, or else $\xi>0$. But in the latter case, substituting  $pu,pv$ for $u,v$ and dividing by $p^{\xi}$, we obtain 
a new cubic surface $(p^{3-\xi}f,\tilde g)$.
If $\bar g$ does not have any $\FF_p$-roots then certainly $(f,g)$ is not soluble over $\QQ_p$. 
Next, suppose that 
$\bar f$ has a triple root and $\bar g$ has no $\FF_p$-roots. If $\xi=0$ then once again the surface $(\bar f,\bar g)$ has a  smooth $\FF_p$-point which can be lifted to ensure solubility over $\QQ_p$.

For each $\xi\in \{0,1,2\}$, 
we define  the following probabilities:
\begin{itemize}
\item
$\alpha_1(\xi)$ is the probability 
of solubility for a pair $(f,p^{\xi}\tilde g)$ such that both $\bar f$ and $\bar g$ have a triple root;
\item
$\alpha_2(\xi)$ is the probability 
of solubility for a pair $(f,p^{\xi}\tilde g)$ such that $\bar f$  has a triple root and 
 $\bar g$ is insoluble;
\item
$\alpha_3(\xi)$ is the probability 
of solubility for a pair $(f,p^{\xi}\tilde g)$  such that 
 $\bar f$ is insoluble and 
$\bar g$  has a triple root;
\item
$\alpha_4(\xi)$ is the probability 
of solubility for a pair $(f,p^{\xi}\tilde g)$ such that $\bar f$ and $\bar g$  are both insoluble.
\end{itemize}
Our work above already implies that 
$$
\alpha_2(0)=\alpha_3(0)=\alpha_4(0)=1 \quad \text{ and } \quad
\alpha_4(1)=\alpha_4(2)=0.
$$
Next we claim that 
$
\alpha_3(\xi)=\alpha_2(3-\xi)$ for $\xi\in \{1,2\}$.
Indeed, over $\FF_p$, the only solutions have $u=v=0$. Substituting $pu$ for $u$ and $pv$ for $v$ and dividing by $p^2$ we are left with the cubic surface $(p^{3-\xi} f,\tilde g)$. This establishes the claim, implying that 
$$
\alpha_2(\xi)+\alpha_3(\xi)=\alpha_2(\xi)+\alpha_2(3-\xi)=
\alpha_2(1)+\alpha_2(2),
$$
for any $\xi\in \{1,2\}$.

We now return to our calculation of $\vartheta_p$. 
For each $1\leq i\leq 4 $, we let $\beta_i(\xi)$
be the probability that our pair of cubic forms $(f,g)=(f,p^\xi \tilde g)$ has the reduction type considered in the definition of $\alpha_i(\xi)$.
Let $S\subset \ZZ_p^4$ be a set of coefficients producing primitive binary cubic forms. 
Then for any  $\xi\in \{0,1,2\}$, the probability that 
a binary cubic form takes the shape $p^{k}\tilde g$ 
for $k\equiv \xi \bmod{3}$, 
where  $\tilde g$ has coefficient vector in $S$,   is 
$$
\sum_{k\equiv \xi\bmod{3}}
\frac{1}{p^{4k}}
\mu(S)=
\frac{1}{p^{4\xi}}\left(1-\frac{1}{p^{12}}\right)^{-1}\mu(S),
$$
where $\mu(S)$ is the probability that a primitive binary cubic form has coefficient vector in  $S$.
It now follows from  Lemma \ref{lem:manjul} that
$\beta_i(\xi)=p^{-4\xi}\beta_i/
(1-1/p^{12})$, 
for $1\leq i\leq 4$, where
\begin{align*}
\beta_1=\tau_*^2, 
\quad
\beta_2=\beta_3=\tau_* (1-\sigma_*-\tau_*),
\quad
\beta_4=(1-\sigma_*-\tau_*)^2.
\end{align*}
Hence
we deduce that 
the probability of insolubility of a cubic surface \eqref{eq:cubics} with coefficients in $\QQ_p$ is equal to
\begin{equation}\label{eq:goal}
\begin{split}
1 - \vartheta_p 
=~& 
\hspace{-0.2cm}
\sum_{\xi\in \{0,1,2\}}
\sum_{1\leq i\leq 4}
\beta_i(\xi)(1 - \alpha_i(\xi))\\
= ~&
\frac{\tau_*^2}{1-1/p^{12}}
\left(1 - \alpha_1(0)
+\frac{1-\alpha_1(1)}{p^4}
+\frac{1-\alpha_1(2)}{p^8}\right)\\
&+
(2- \alpha_2(1)- \alpha_2(2))
\frac{(1+1/p^4)\tau_*(1-\sigma_*-\tau_*)}{p^4(1-1/p^{12})}\\
 &+
  \frac{(1-\sigma_*-\tau_*)^2(1+1/p^4)}{p^4(1-1/p^{12})}.
\end{split}
\end{equation}
It remains to analyse $\alpha_1(\xi)$ for $\xi\in \{0,1,2\}$ and $\alpha_2(1), \alpha_2(2)$. In doing so, 
when the reduction of the form has a triple point, 
the density will not depend on what this triple point is and so  
we shall always take it to be $(1 : 0)$.

In the special case $\alpha_1(1)$, which corresponds to surfaces $(f,p\tilde g)$, where
$\bar f$ and $\bar g$ both have triple roots, we may  assume that $\bar f$ and $\bar g$ both have the triple root $(1:0)$. 
Substituting $pv$ for $v$ and dividing by $p$, it is easily seen  that 
 $\alpha_1(1)$ is equal to the probability of solubility over $\QQ_p$ for a cubic surface with valuations 
\begin{align*}
&\geq 0 ~\geq 1 ~\geq 2 ~= 2\\
&\geq 1 ~\geq 1 ~\geq 1 ~= 0,
\end{align*}
in which one is only interested in solutions with  $\min\{v_p(u),v_p(x)\}=0$.

As in \cite{local}, we will  develop recursions in order to solve for the densities of soluble
cubics among those whose reductions produce the bad configurations outlined above.

\subsection{Technical lemmas}

The purpose of this section is to collect together definitions and calculations of particular probabilities that will feature in our analysis.

\begin{lemma}\label{lem:lambda}
Let 
$\lambda$ be 
the probability that a cubic surface $(f,g)$  with  valuations 
\begin{align*}
&=  2 ~\geq 2 ~\geq 2 ~= 1\\
&\geq 0 ~\geq 0 ~\geq 1 ~= 1,
\end{align*}
is soluble over $\ZZ_p$ with $x\in \ZZ_p^*$. 
Then 
$$
\lambda=
\left(1-\frac{1-\sigma_2}{p^5}\right)^{-1}
\left(
1-\frac{1}{p}+\frac{1}{p^2}-\frac{(1-\sigma_2)^2}{p^4}
-\frac{(1-\sigma_2)\sigma_2}{p^5}\right).
$$
\end{lemma}
\begin{proof}
If $v_p(d_1)=0$ (which happens with probability $1-1/p$) there is a simple root which can be lifted to a $\QQ_p$-point on the cubic surface with $x\in \ZZ_p^*$.
 If $v_p(d_0)=0$ and $v_p(d_1)\geq 1$ (probability $(p-1)/p^2$) then the equation is insoluble over $\QQ_p$ with $x\in \ZZ_p^*$.
If $v_p(d_0), v_p(d_1)\geq 1$ (probability $1/p^2$) then we divide by $p$ to get the valuations 
\begin{align*}
&=  1 ~\geq 1 ~\geq 1 ~= 0\\
&\geq 0 ~\geq 0 ~\geq 0 ~= 0.
\end{align*}
With probability $\sigma_1$ we obtain a simple root of $g$ over $\FF_p$ which can be lifted. Note that if $\bar g$ is insoluble (probability $1-\sigma_1-\tau_1$) then the equation $\bar f=\bar g$ is absolutely irreducible and so has a smooth $\FF_p$-point which can be lifted. 
Finally,  with probability $\tau_1$ there is a triple root over $\FF_p$, which we may take to be $(1:0)$.  
Now with probability $\sigma_2$ the surface is soluble over $\ZZ_p$ with $x\in \ZZ_p^*$.
With probability $1-\sigma_2$  the diagonal cubic form is insoluble over $\FF_p$ and we replace $v$ by $pv$ and $y$ by $pv$. Dividing by $p$ leads us 
to the valuations
\begin{align*}
&=  0 ~\geq 1 ~\geq 2 ~= 2\\
&\geq 0 ~\geq 1 ~\geq 2 ~= 2.
\end{align*}
If $v_p(d_0)=0$ (probability $1-1/p$) then we get solubility over $\QQ_p$ 
with probability $\sigma_2$, since when the diagonal cubic form is insoluble over $\FF_p$ we will not obtain solutions with $x\in \ZZ_p^*$. Alternatively, if $v_p(d_0)\geq 1$ (probability $1/p$)
then we replace $u$ by $pu$ and divide once more by $p$ to get the valuations
\begin{align*}
&=  2 ~\geq 2 ~\geq 2 ~= 1\\
&\geq 0 ~\geq 0 ~\geq 1 ~= 1,
\end{align*}
for which 
the probability of solubility over $\QQ_p$ with $x\in \ZZ_p^*$ is 
$\lambda$.
All in all it follows that 
\begin{align*}
\lambda
&=1-\frac{1}{p}+\frac{1}{p^2}\left\{1-\tau_1+ \tau_1\left(\sigma_2+(1-\sigma_2)\left(\frac{(p-1)\sigma_2}{p}+\frac{\lambda}{p}\right)\right)\right\}.
%&=1-\frac{1}{p}+\frac{1-(1-\sigma_2)^2\tau_1}{p^2}
%-\frac{(1-\sigma_2)\sigma_2\tau_1}{p^3}+
%\frac{(1-\sigma_2)\tau_1\lambda}{p^3}.
\end{align*}
The lemma follows on rearranging and recalling that $\tau_1=1/p^2$.
\end{proof}

\begin{lemma}\label{lem:rho}
Let 
$\rho$ be 
the probability that a cubic surface $(f,g)$  with  valuations 
\begin{align*}
&=  0 ~\geq 1 ~\geq 1 ~= 1\\
&\geq 0 ~\geq 1 ~\geq 2 ~= 2,
\end{align*}
is soluble over $\ZZ_p$ with 
$x\in \ZZ_p^*$.
Then 
$$
\rho=
\left(1-\frac{1}{p^5}\right)^{-1}\left(
\left(1-\frac{1}{p}+\frac{1}{p^2}-\frac{1}{p^3}\right) \sigma_2+\frac{1}{p}
-\frac{1}{p^2}+\frac{\sigma_1}{p^3}\right).
$$
\end{lemma}

\begin{proof}
The first case to consider is when  $v_p(d_0)=0$ (which happens with probability $(p-1)/p$). In this case the cubic surface is soluble over $\QQ_p$ 
with $x\in \ZZ_p^*$
with probability $\sigma_2$, else it is insoluble over $\QQ_p$. 
Suppose next that  $v_p(d_0)\geq 1$ (probability $1/p$).
Then we must have $v_p(u)\geq 1$. Replacing $u$ by $pu$ and simplifying, we get a cubic surface with  valuations
\begin{align*}
&=  2 ~\geq 2 ~\geq 1 ~= 0\\
&\geq 0 ~\geq 0 ~\geq 1 ~= 1.
\end{align*}
When $v_p(d_1)=0$ (probability $(p-1)/p$) this is soluble over $\QQ_p$. 
When $v_p(d_0)=0$ and $v_p(d_1)\geq 1$ (probability $(p-1)/p^2$) this is soluble over  $\QQ_p$ 
with probability $\sigma_2$. Finally, we suppose that $v_p(d_0),v_p(d_1)\geq 1$ (probability $1/p^2$). In this case we must have 
$v_p(v)\geq 1$. Replacing $v$ by $pv$ and simplifying we get the valuations
\begin{align*}
&=  1 ~\geq 2 ~\geq 2 ~= 2\\
&\geq 0 ~\geq 0 ~\geq 0 ~= 0.
\end{align*}
With probability $\sigma_1$ the second cubic form has a simple root which can be lifted to get solubility over $\QQ_p$. With probability $\tau_1$ we get a triple root, which on moving to $(1:0)$ leads us to replace $y$ by $py$. On simplification, this gives a cubic surface with valuations whose probability of solubility is $\rho.$

All in all, it follows that 
\begin{align*}
\rho=~& \frac{(p-1)\sigma_2}{p}+\frac{1}{p}
\left\{\frac{p-1}{p}+
\frac{(p-1)\sigma_2}{p^2}+\frac{\sigma_1+\tau_1\rho}{p^2}
\right\}.
\end{align*}
Rearranging this and recalling that $\tau_1=1/p^2$, 
the lemma follows.
\end{proof}

\begin{lemma}\label{lem:omega0}
Let 
$\omega$ be 
the probability that a cubic surface $(f,g)$  with  valuations 
\begin{align*}
&=   0 ~\geq 1 ~\geq 2 ~= 2\\
&\geq 1 ~\geq 1 ~\geq 1 ~= 0,
\end{align*}
is soluble over $\ZZ_p$ with 
$\min\{v_p(u),v_p(x)\}=0$.
Then 
$$
\omega=\left(1-\frac{1-\sigma_2}{p^5}\right)^{-1}\left(
\sigma_2+\frac{1-\sigma_2}{p}\left(1-\frac{1-\sigma_2}{p}+\frac{1-\sigma_2}{p^2}-\frac{1}{p^4}\right)
\right).
$$
\end{lemma}

\begin{proof}
With probability $\sigma_2$ there is a solution over $\FF_p$ which can be lifted to ensure solubility over $\QQ_p$. Otherwise, with probability $1-\sigma_2$, 
the only solution has $v_p(u),v_p(y)\geq 1$. Replacing $u$ by $pu$, $y$ by $py$ and simplifying, we arrive at the valuations
\begin{align*}
&=   2 ~\geq 2 ~\geq 2 ~= 1\\
&\geq 0 ~\geq 1 ~\geq 2 ~= 2,
\end{align*}
in which we are interested in solubility over $\QQ_p$ with $x\in \ZZ_p^*$.
If $v_p(d_0)=0$ (probability $(p-1)/p$) then there are no solutions. If, on the other hand, $v_p(d_0)\geq 1$ then we simplify to get the valuations
\begin{align*}
&=   1 ~\geq 1 ~\geq 1 ~= 0\\
&\geq 0 ~\geq 0 ~\geq 1 ~= 1.
\end{align*}
Now if $v_p(d_1)=0$ (probability $(p-1)/p$)
then there is a simple $\FF_p$-root of $g$  which can be lifted to get solubility over $\QQ_p$. If $v_p(d_1)\geq 1$ and $v_p(d_0)=0$ 
(probability $(p-1)/p^2$)
then the cubic surface is soluble over $\QQ_p$ with probability  $\sigma_2$.
If $v_p(d_0),v_p(d_1)\geq 1$ (probability $1/p^2$) then we replace $v$ by $pv$ and simplify, in order to get the valuations
\begin{align*}
&=   0 ~\geq 1 ~\geq 2 ~= 2\\
&\geq 0 ~\geq 0 ~\geq 0 ~= 0.
\end{align*}
With probability $\sigma_1$ the second polynomial has a simple root over $\FF_p$ which can be lifted to get solubility over $\QQ_p$. With probability $1-\sigma_1-\tau_1$ the form $\bar g$ is irreducible over $\FF_p$ and 
we also get solubility over $\QQ_p$. Finally, with probability $\tau_1$, the form $\bar g$ has a triple root. Without loss of generality we may assume that the root is 
$(1:0)$ and the probability of solubility over $\QQ_p$ for a cubic surface with these valuations is $\omega$. 

Putting everything together, we obtain
$$
\omega=
\sigma_2+\frac{1-\sigma_2}{p}
\left(
1-\frac{1}{p}+
\left(1-\frac{1}{p}\right)\frac{\sigma_2}{p}+
\frac{1-\tau_1+\tau_1\omega}{p^2}
\right).
$$
The lemma follows on rearranging  and recalling that $\tau_1=1/p^2$.
\end{proof}

\begin{lemma}\label{lem:gamma}
For  $i\in \{0,1\}$  let 
$\gamma_i$ be 
the probability that a cubic surface $(f,g)$
is soluble over $\ZZ_p$ with $x\in \ZZ_p^*$, 
assuming that 
 $\bar f$ is irreducible over $\FF_p$ and the surface has
valuations 
\begin{align*}
&\geq   0 ~\geq 0 ~\geq 0 ~= 0\\
&\geq  0 ~\geq i ~\geq i+1 ~= i+1.
\end{align*}
Then 
$$
\gamma_i=
\left(1-\frac{1}{p^5}\right)^{-1}\left(1-\frac{1}{p^2}+\frac{\sigma_1}{p^{2+i}}
\right).
$$
\end{lemma}

\begin{proof}
If $v_p(d_0)=0$  (probability $1-1/p$)
then the cubic surface is soluble over $\QQ_p$. 
If $v_p(d_0)\geq 1$ and $v_p(d_1)=0$ (probability $(1-i)(1-1/p)/p$) then 
we also get solubility over $\QQ_p$ with $x\in \ZZ_p^*$.
If  $v_p(d_0),v_p(d_1)\geq 1$  (probability $(1+i(p-1))/p^2$) then we replace $u$ by $pu$, $v$ by $pv$ and simplify to get the valuations
\begin{align*}
&\geq   2 ~\geq 2 ~\geq 2 ~= 2\\
&\geq  0 ~\geq 0 ~\geq i ~= i.
\end{align*}
We seek solubility over $\QQ_p$ with $x\in \ZZ_p^*$. 

Suppose first that $i=0$. Then with 
probability $\sigma_1$ the second cubic form $g$ has a simple $\FF_p$-root which can be lifted. With probability $1-\sigma_1-\tau_1$ it is not soluble over $\QQ_p$ and with probability $\tau_1$ the form $g$ has a triple root, which we can move  to $(1:0).$  Replacing $y$ by $py$ we are led to the new valuations
\begin{align*}
&\geq   1 ~\geq 1 ~\geq 1 ~= 1\\
&\geq  0 ~\geq 1 ~\geq 2 ~= 2.
\end{align*}
If $v_p(d_0)=0$ then there are no solutions. If $v_p(d_0)\geq 1$ (probability $1/p$) then we simplify and observe that the probability of solubility is now $\gamma_0$. 

Suppose next that $i=1$.
If $v_p(d_1)=0$ (probability $1-1/p$) then we get solubility over $\QQ_p$. If $v_p(d_1)\geq 1$ and $v_p(d_0)=0$ then the surface is not soluble over $\QQ_p$ with $x\in \ZZ_p^*$. If $v_p(d_0),v_p(d_1)\geq 1$ (probability $1/p^2$) then we simplify to arrive at the valuations
\begin{align*}
&\geq   1 ~\geq 1 ~\geq 1 ~= 1\\
&\geq  0 ~\geq 0 ~\geq 0 ~= 0.
\end{align*}
With probability $\sigma_1$ the second cubic form $g$ has a simple $\FF_p$-root which can be lifted to ensure the desired solubility over $\QQ_p$. With probability $1-\sigma_1-\tau_1$ it is not soluble over $\QQ_p$ and with probability $\tau_1$ the form $g$ has a triple root, which we can move  to $(1:0).$  Replacing $y$ by $py$ we are led to the a cubic surface for which the probability of 
solubility with $x\in \ZZ_p^*$ is precisely $\gamma_1$.

Thus, for $i\in \{0,1\}$, we have shown that 
$$
\gamma_i=1-\frac{1}{p^2}-
\frac{i(p-1)}{p^2}+
\frac{1+i(p-1)}{p^2}\times
\begin{cases}
\sigma_1+\frac{\tau_1\gamma_0}{p} & \text{ if $i=0$,}\\
1-\frac{1}{p}+\frac{\sigma_1+\tau_1\gamma_1}{p^2} & \text{ if $i=1$.}
\end{cases}
$$
The lemma follows on rearranging and recalling that $\tau_1=1/p^2$.
\end{proof}

\begin{lemma}\label{lem:epsilon}
Let 
$\epsilon$ be 
the probability that a cubic surface $(f,g)$  with  valuations 
\begin{align*}
&\geq 0 ~\geq 0 ~\geq 1 ~= 1\\
&\geq   1 ~\geq 1 ~\geq 1 ~= 0,
\end{align*}
is soluble over $\ZZ_p$ with 
$\min\{v_p(u),v_p(x)\}=0$.
Then $\epsilon=A+\alpha_1(1)/p^4$, where
\begin{align*}
A=~&1-\frac{1}{p}+
\frac{p-1}{p^2}\left(\sigma_2 +(1-\sigma_2)\left(
\left(1-\frac{1}{p}\right)\sigma_2
+\frac{\lambda}{p}\right)\right)\\ &+\frac{\sigma_1+(1-\sigma_1)\gamma_1}{p^2}- \frac{\gamma_1}{p^4}.
\end{align*}
\end{lemma}

\begin{proof}
If $v_p(c_1)=0$ (probability $(p-1)/p$) then the surface is soluble over $\QQ_p$.

If $v_p(c_1)\geq 1$ and $v_p(c_0)=0$ (probability $(p-1)/p^2$) then with probability $\sigma_2$ the surface is soluble over $\QQ_p$. Alternatively (probability $1-\sigma_2$) we must have $v_p(u),v_p(y)\geq 1$ in any solution. Replacing $u$ by $pu$, $v$ by $pv$ and simplifying, we are led to the valuations
\begin{align*}
&=  2 ~\geq 2 ~\geq 1 ~= 0\\
&\geq 0 ~\geq 1 ~\geq 2 ~= 2,
\end{align*}
in which we are interested in solubility over $\QQ_p$ with 
$x\in \ZZ_p^*$.
If $v_p(d_0)= 0$ (probability $(p-1)/p$) then this is soluble over $\QQ_p$ with probability $\sigma_2$. If $v_p(d_0)\geq 1$ (probability $1/p$) 
then we replace $v$ by $pv$ and simplify to get  valuations for which 
the probability of solubility for such cubic surfaces is precisely $\lambda$, in the notation of Lemma \ref{lem:lambda}.

We now turn to the case  $v_p(c_1),v_p(c_0)\geq 1$ (probability $1/p^2$) in the original equation.  Replacing $y$ by $py$ and simplifying we arrive at the valuations
\begin{align*}
&\geq 0 ~\geq 0 ~\geq 0 ~= 0\\
&\geq   0 ~\geq 1 ~\geq 2 ~= 2,
\end{align*}
in which we are interested in solubility over $\QQ_p$ with 
$\min\{v_p(u),v_p(x)\}=0$.
With probability $\sigma_1$ the first cubic form has a simple root which can be lifted. With 
 probability $1-\sigma_1-\tau_1$ the first cubic form is irreducible  $\FF_p$
and the probability of solubility is equal to $\gamma_1$, in the notation of Lemma \ref{lem:gamma}. Finally, with probability $\tau_1$ it has a triple root, which we may move $(1:0).$ In this final case, in view of our earlier discussions, the probability of solubility is equal to $\alpha_1(1)$. 
The lemma easily follows on recalling that $\tau_1=1/p^2$.
\end{proof}

\subsection{Calculation of $\alpha_2(1)$ and $\alpha_2(2)$}

For $\xi\in \{1,2\}$ we proceed to calculate  $\alpha_2(\xi)$.
The density doesn't depend on the triple point that $\bar f$ has and so we may take it to be $(1:0)$. Substituting $pv$ for $v$ and dividing by $p$ we arrive at the cubic surface $(f',p^{\xi-1}\tilde g)$ where the reduction $\bar g$ 
of  $\tilde g$ modulo $p$  is insoluble and the coefficients of $f'$ have valuations
$$
\geq 0 ~\geq 1~\geq 2 ~=2.
$$

If $\xi=1$ then clearly $\alpha_2(1)=\gamma_1$, in the notation of Lemma \ref{lem:gamma}. Suppose next that  $\xi=2$. If $v_p(c_0)=0$ then we must  have $v_p(u)\geq 1$ and a change of variables from $u$ to $pu$ leads us to the conclusion that the surface is insoluble over $\QQ_p$ since 
 $\bar g$ is insoluble over $\FF_p$.   
 If $v_p(c_0)\geq 1$ (probability $1/p$) then we divide by $p$ to get the cubic surface $(f'',\tilde g)$, in which the coefficients of $f''$ have valuations
$$
\geq 0 ~\geq 0~\geq 1 ~=1.
$$
The probability of solubility in this case is precisely $\gamma_0$, in the notation of Lemma \ref{lem:gamma}. Thus we have shown that
$$
\alpha_2(1)=\gamma_1 \quad \text{ and } \quad
\alpha_2(2)=\frac{\gamma_0}{p}.
$$

\subsection{Calculation of $\alpha_1(0)$}

We may assume that the coefficients of $(f,g)$ have valuations 
\begin{align*}
&\geq  1 ~\geq 1 ~\geq 1 ~= 0\\
&\geq 1 ~\geq 1 ~\geq 1 ~= 0.
\end{align*}
With probability $\sigma_2$ the equation $c_3v^3=d_3y^3$ has a simple root over $\FF_p$ which can lifted, thereby ensuring solubility over $\QQ_p$. 
With probability $1-\sigma_2$ it is insoluble over $\FF_p$. In the latter case we replace $v$ by $pv$, $y$ by $py$  and simplify.

Suppose first that $\min\{v_p(c_0),v_p(d_0)\}\geq 2$ (which happens with probability $1/p^2$). Then we get a new pair of cubic forms with valuations
\begin{align*}
&\geq  0 ~\geq 0 ~\geq 1 ~= 1\\
&\geq 0 ~\geq 0 ~\geq 1 ~= 1.
\end{align*}
If $v_p(c_1)=0$ or $v_p(d_1)=0$ (probability $1-1/p^2$) then there is a simple root over $\FF_p$ which can be lifted. If, on the other hand, 
$v_p(c_1), v_p(d_1)\geq 1$ (probability $1/p^2$) then the probability of solubility is $\kappa$, where $\kappa$ denotes the probability that a pair of cubics with valuations 
\begin{align*}
&\geq  0 ~\geq 1 ~\geq 1 ~= 1\\
&\geq 0 ~\geq 1 ~\geq 1 ~= 1
\end{align*}
is soluble over $\ZZ_p$ with $\min\{v_p(u),v_p(x)\}=0$.
We claim that 
\begin{align*}
\kappa
&=
\left(1-\frac{1}{p}\right)^2\sigma_2+\frac{1}{p^2}-\frac{1}{p^6}+\frac{\alpha_1(0)}{p^6}+\frac{2\kappa_1}{p}\left(1-\frac{1}{p}\right),
\end{align*}
where $\kappa_1$ is the probability of solubility for a cubic surface with valuations
\begin{align*}
&=  0 ~\geq 1 ~\geq 1 ~= 1\\
&\geq 1 ~\geq 1 ~\geq 1 ~= 1,
\end{align*}
with $x\in \ZZ_p^*$.
To see this, 
if $v_p(c_0)=v_p(d_0)=0$ (probability $(1-1/p)^2$), then with probability $\sigma_2$ we get a simple root over $\FF_p$ which can be lifted (and it is insoluble over $\QQ_p$ otherwise). 
If $v_p(c_0),v_p(d_0)\geq 1$ (probability $1/p^2$) then we simplify and obtain  the valuations 
\begin{align*}
&\geq  0 ~\geq 0 ~\geq 0 ~= 0\\
&\geq 0 ~\geq 0 ~\geq 0 ~= 0.
\end{align*}
With probability $\sigma_1$ the cubic form $f$ has a simple $\FF_p$-root which can be lifted to get solubility over $\QQ_p$. With probability $1-\sigma_1-\tau_1$ the form $f$ is irreducible over $\FF_p$ and we also get solubility over $\QQ_p$. Finally, with probability $\tau_1$ the form $f$ has a triple root. In this case, with probability 
$1-\tau_1$ we get solubility over $\QQ_p$, since either $g$ has a simple root over $\FF_p$ or it is irreducible over $\FF_p$.  
Alternatively, with probability $\tau_1$ the form $g$ also has a triple root over $\FF_p$ and the probability of solubility is precisely $\alpha_1(0)$. The claim follows on 
recalling that $\tau_1=1/p^2$ and 
noting that the probability of solubility is $\kappa_1$ when 
$v_p(c_0)=0, v_p(d_0)\geq 1$ or  
$v_p(c_0)\geq 1, v_p(d_0)=0$.

Next we calculate $\kappa_1$, noting that we must have $v_p(u)\geq 1$ in any solution. Replacing $u$ by $pu$ and simplifying we arrive at the valuations 
\begin{align*}
&=  2 ~\geq 2 ~\geq 1 ~= 0\\
&\geq 0 ~\geq 0 ~\geq 0 ~= 0.
\end{align*}
With probability $1-\tau_1$ this is soluble over $\QQ_p$. With probability $\tau_1$ the second form $g$ has a triple root over $\FF_p$ which we can move to $(1:0)$. But then with probability $\sigma_2$ the cubic surface is soluble over $\QQ_p^*$ and with probability $1-\sigma_2$ the only solution has $v_p(v),v_p(y)\geq 1$. Replacing $v$ by $pv$, $y$ by $py$ and simplifying, we therefore obtain the valuations
\begin{align*}
&=  1 ~\geq 2 ~\geq 2 ~= 2\\
&\geq 0 ~\geq 1 ~\geq 1 ~= 1.
\end{align*}
If $v_p(d_0)=0$ there are no solutions over $\QQ_p^*$, while if $v_p(d_0)\geq 1$ we may remove a factor of $p$ from all of the coefficients. In the new pair of forms we get solubility if $v_p(d_1)=0$ (probability $1-1/p$), while we get solubility with probability $\sigma_2$ if $v_p(d_1)\geq 1$ and $v_p(d_0)=0$
(probability $(p-1)/p^2$). Finally, the probability of solubility is $\kappa_1$ if $v_p(d_1),v_p(d_0)\geq 1$ (probability $1/p^2$). 
Recalling that $\tau_1=1/p^2$ and rearranging, we easily conclude that 
$$
\kappa_1=\left(1-\frac{1-\sigma_2}{p^5}\right)^{-1}\left(1-\frac{1-\sigma_2}{p^2}+\frac{1-\sigma_2}{p^3}\left(
1-\frac{1-\sigma_2}{p}-\frac{\sigma_2}{p^2}
\right)\right).
$$

Next, suppose that $v_p(c_0)=1$ and $v_p(d_0)\geq 2$ (probability $(p-1)/p^2$). Then our cubic surface has valuations
\begin{align*}
&=  0 ~\geq 1 ~\geq 2 ~= 2\\
&\geq 1 ~\geq 1 ~\geq 2 ~= 2.
\end{align*}
Hence we replace $u$ by $pu$ and divide by $p$ to get the new valuations
\begin{align*}
&=  2 ~\geq 2 ~\geq 2 ~= 1\\
&\geq 0 ~\geq 0 ~\geq 1 ~= 1.
\end{align*}
But the probability that a pair of cubics with these valuations 
is soluble over $\ZZ_p$ with $x\in \ZZ_p^*$ is precisely equal to $\lambda$ in the notation of Lemma \ref{lem:lambda}. 
The case in which 
$v_p(c_0)\geq 2$ and $v_p(d_0)=1$ yields the same conclusion.

Finally, we suppose that  $v_p(c_0)=v_p(d_0)=1$ (probability $(p-1)^2/p^2$). This leads to the valuations
\begin{align*}
&=  0 ~\geq 1 ~\geq 2 ~= 2\\
&= 0 ~\geq 1 ~\geq 2 ~= 2.
\end{align*}
With probability $\sigma_2$ the equation $c_0u^3=d_0x^3$ has a simple root over $\FF_p$ which can be lifted. The alternative case leads to insolubility over $\QQ_p$.

Summarising our argument so far, we have shown that 
$$
\alpha_1(0)=\sigma_2+(1-\sigma_2)\left\{\frac{1}{p^2}\left(1+\frac{\kappa-1}{p^2} \right)+
\frac{2(p-1)\lambda}{p^2} +\left(1-\frac{1}{p}\right)^2\sigma_2
 \right\}.
$$
Our calculations show that 
$\kappa=B+\alpha_1(0)/p^6$, where
\begin{align*}
B=~&
\left(1-\frac{1}{p}\right)^2\sigma_2+\frac{1}{p^2}-\frac{1}{p^6}\\
&+\frac{2\left(p-1\right)}{
p^2(1-\frac{1-\sigma_2}{p^5})}\left(1-\frac{1-\sigma_2}{p^2}+\frac{1-\sigma_2}{p^3}\left(
1-\frac{1-\sigma_2}{p}-\frac{\sigma_2}{p^2}
\right)\right).
\end{align*}
Substituting $\kappa$ and rearranging, we finally conclude that 
\begin{align*}
\alpha_1(0)=~&\left(1-\frac{1-\sigma_2}{p^{10}}\right)^{-1}\\
&\times
\left(
\sigma_2+(1-\sigma_2)
\left\{
\frac{1}{p^2}+\frac{B-1}{p^4}+\frac{2\lambda}{p}\left(1-\frac{1}{p}\right)+
\left(1-\frac{1}{p}\right)^2\sigma_2
\right\}
\right).
\end{align*}

\subsection{Calculation of $\alpha_1(1)$}

In this section we calculate the 
value of $\alpha_1(1)$ recorded at the close of \S \ref{s:local-er}.
If $v_p(c_0)=0$ (probability $1-1/p$)  then
the probability of solubility is precisely $\omega$, in the notation of Lemma \ref{lem:omega0}. Alternatively, with probability $1/p$, we have 
$v_p(c_0)\geq 1$. Replacing $y$ by $py$ and simplifying, we
deduce that 
$$
\alpha_1(1)=\left(1-\frac{1}{p}\right)\omega+\frac{\omega_1}{p},
$$
were $\omega_1$ is 
 the probability of solubility  over $\QQ_p$ for a cubic surface with valuations 
\begin{align*}
&\geq   0 ~\geq 0 ~\geq 1 ~= 1\\
&\geq 0 ~\geq 1 ~\geq 2 ~= 2,
\end{align*}
with $\min\{v_p(u),v_p(x)\}=0$.

It remains to calculate $\omega_1$. 
If $v_p(c_1)=0$ (probability $(p-1)/p$) then there is a simple root over $\FF_p$ which can lifted. 
If $v_p(c_1)\geq 1$ and $v_p(c_0)=0$ (probability $(p-1)/p^2$) then the probability of solubility over $\QQ_p$ is precisely $\rho$, in the notation of Lemma \ref{lem:rho}.
Suppose now that $v_p(c_1),v_p(c_0)\geq 1$ (probability $1/p^2$).
If  $v_p(d_0)=0$ (probability $(p-1)/p$) then we must replace $x$ by $px$ in order to arrive at the valuations
\begin{align*}
&\geq   0 ~\geq 0 ~\geq 0 ~= 0\\
&= 2 ~\geq 2 ~\geq 2 ~= 1.
\end{align*}
Arguing as in the close of the proof of Lemma \ref{lem:rho}, the probability of solubility over $\ZZ_p$ with $u\in \ZZ_p^*$ is found to be $\sigma_1+\tau_1\rho$.

If $v_p(d_0)\geq 1$ (probability $1/p$), we simplify to get the valuations
\begin{align*}
&\geq   0 ~\geq 0 ~\geq 0 ~= 0\\
&\geq 0 ~\geq 0 ~\geq 1 ~= 1.
\end{align*}
We are interested in solubility over $\QQ_p$ with $\min\{v_p(u),v_p(x)\}=0$.
With probability $\sigma_1$ the first form has a simple $\FF_p$-root which can be lifted. With probability $1-\sigma_1-\tau_1$ the first form is irreducible over $\FF_p$ and the overall probability of solubility is equal to 
$\gamma_0$, in the notation of Lemma \ref{lem:gamma}.
Finally, with probability $\tau_1$ the first form has a triple root which can be moved to $(1:0)$. In this case the probability of solubility is equal to $\epsilon=A+\alpha_1(1)/p^4$, in the notation of Lemma \ref{lem:epsilon}.

Recalling that $\tau_1=1/p^2$,  we have shown that 
$
\omega_1=
C+\alpha_1(1)/p^9,
$
where
$$
C=
1-\frac{1}{p}+\frac{(p-1)\rho}{p^2}+
\frac{(p-1)(\sigma_1+\rho/p^2)}{p^3}+\frac{\sigma_1+(1-\sigma_1-1/p^2)\gamma_0}{p^3}+
\frac{A}{p^5}.
$$
We now  substitute this into our earlier expression and rearrange things to finally conclude that 
$$
\alpha_1(1)=\left(1-\frac{1}{p^{10}}\right)^{-1}\left\{\left(1-\frac{1}{p}\right)\omega+\frac{C}{p}\right\}.
$$

\subsection{Calculation of $\alpha_1(2)$}

It remains to study the probability of solubility over $\QQ_p$ for 
cubic surfaces with valuations 
\begin{align*}
&\geq   1 ~\geq 1 ~\geq 1 ~= 0\\
&\geq 3 ~\geq 3 ~\geq 3 ~= 2.
\end{align*}
We claim that 
$$
\alpha_1(2)=
\left(1-\frac{1}{p}\right)\rho+\frac{A}{p}+\frac{\alpha_1(1)}{p^5},
$$
in the notation of Lemmas \ref{lem:rho} and \ref{lem:epsilon}.
To begin with, 
if $v_p(c_0)=1$ (which happens with probability $(p-1)/p$) then we must have $v_p(u), v_p(v), v_p(y)\geq 1$ in any solution. Making the obvious transformations and simplifying, the probability of solubility over $\QQ_p$ with $x\in \ZZ_p^*$  is precisely $\rho$.  
If, on the other hand,  $v_p(c_0)\geq 2$ (probability $1/p$) then we replace $v$  by $pv$ to arrive at valuations for which 
 the probability of solubility over $\QQ_p$ is  $\epsilon$, in the notation of Lemma~\ref{lem:epsilon}. The claim follows.

\subsection{Conclusion}

Returning to \eqref{eq:goal}, we substitute the values of $\sigma_*, \tau_*$ from part (2) of Lemma 
\ref{lem:manjul}. This leads to the expression
\begin{align*}
\chi_p(1/p)
= ~&
\frac{E(p)}{p^4(1+1/p^2)^2(1-1/p^{12})}
\end{align*}
where
\begin{align*}
E(p)=~& 1 - \alpha_1(0)
+\frac{1-\alpha_1(1)}{p^4}
+\frac{1-\alpha_1(2)}{p^8}\\
&+
\frac{(1-1/p)(1+1/p^4)}{3p^2}
(2- \alpha_2(1)- \alpha_2(2))
 +
  \frac{(1-1/p)^2(1+1/p^4)}{9}.
\end{align*}
Using a computer algebra package, it remains to substitute our various expressions for $\alpha_1(0),\alpha_1(1), \alpha_1(2)$ and $\alpha_2(1), \alpha_2(2)$ into this, 
before  simplifying the expression. 
This ultimately leads  to the expressions for $\chi_p(x)$ recorded in \S \ref{s:local-er}, with $x=1/p$.

\section{Global solubility}\label{s:global}

We have seen that in order to prove Theorem \ref{t:main}
it suffices to prove that there is a positive proportion of irreducible $f\in V(\ZZ)$, when ordered by height, such that  $C_f(\QQ)\neq \emptyset$, where $C_f$ is the associated genus $1$ curve \eqref{eq:Cf}. Our template for proving this is the argument developed by Bhargava \cite[\S 4]{manjul} to show that a positive proportion of  plane cubic curves have a $\QQ$-point.

\subsection{Selmer groups of Mordell curves}\label{s:selmer}

Our work relies on the  arithmetic of 
Mordell curves
$$
E_k:\quad y^2=x^3+k,
$$
for $k\in \ZZ$.  These 
elliptic curves have 
$j$-invariant $0$ and (potential)
 complex multiplication by the ring of integers  of $\QQ(\sqrt{-3})$.
There is a $3$-isogeny $\phi=\phi_k:E_k\to E_{-27k}$ and, by duality, a $3$-isogeny $\hat \phi=\phi_{-27k}: E_{-27k}\to E_k$.
As $k$ varies over the integers, 
Bhargava, Elkies and Schnidman \cite{BES} have undertaken an extensive investigation into the average size of the $\phi$-Selmer groups
$
\sel_{\phi}(E_{k})\subset \H^1(G_\QQ, E_k[\phi])$ and 
$\sel_{\hat \phi}(E_{-27k})\subset \H^1(G_\QQ, E_{-27k}[\hat \phi])$, 
where $G_{\QQ}=\Gal(\bar\QQ/\QQ)$, 
consisting of ``locally soluble'' cohomology classes.
Let 
\begin{align*}
\rho&=\frac{103\cdot 229}{2 \cdot 3^2 \cdot 7^2\cdot 13 }\prod_{p\equiv 5\bmod{6}}
\frac{(1-\frac{1}{p})(1+\frac{1}{p}+\frac{5}{3p^2}+\frac{1}{p^3}+\frac{5}{3p^4}+\frac{1}{p^5})}{1-\frac{1}{p^6}}=2.1265\dots.
\end{align*}
While it doesn't play a role in our work, 
when the elliptic curves $E_{k}$ are ordered by height, 
it follows from  \cite[Thm.~1]{BES} that the average of 
$\#\sel_\phi(E_{k})$ is 
$1+\rho$ (resp.~$1+\rho/3$) if  $k<0$ (resp.~$k>0$).

We shall be interested in soluble elements of $\sel_\phi(E_{k})$ and $\sel_{\hat \phi}(E_{-27k})$
According to the fundamental short exact sequence
$$
0\to E_{-27k}(\QQ)/\phi(E_k(\QQ))\to \sel_{\phi}(E_k)\to \Sha(E_k)[\phi]\to 0,
$$
the soluble elements of $\sel_\phi(E_{k})$ are precisely those that come from the group
$E_{-27k}(\QQ)/\phi(E_k(\QQ))$. Similarly, 
soluble elements of $\sel_{\hat \phi}(E_{-27k})$ are precisely those that come from 
$E_{k}(\QQ)/\hat \phi(E_{-27}(\QQ))$.  
Noting that $\hat \phi \circ \phi =[3]$, we may now  turn to the exact sequence 
$$
0 \to \frac{E_{-27k}(\QQ)[\hat{\phi}]}{\phi(E_k(\QQ)[3])} \to \frac{E_{-27k}(\QQ)}{\phi(E_k(\QQ))} \overset{\hat\phi}{\to} \frac{E_k(\QQ)}{3E_k(\QQ)} \to 
\frac{E_k(\QQ)}{\hat{\phi}(E_{-27k}(\QQ))} \to 0,
$$
which is recorded in \cite[Remark X.4.7]{silverman}. The groups appearing in this sequence 
are all  $\FF_3$-vector spaces and it follows that 
\begin{equation}\label{eq:home}
\begin{split}
\dim_{\FF_3}\frac{E_{-27k}(\QQ)}{\phi(E_k(\QQ))} 
+
\dim_{\FF_3} \frac{E_k(\QQ)}{\hat{\phi}(E_{-27k}(\QQ))} &\geq 
\dim_{\FF_3}  \frac{E_k(\QQ)}{3E_k(\QQ)} \\ &\geq \rank E_k(\QQ). 
\end{split}
\end{equation}
(Since the two groups on the left hand side sit inside the relevant Selmer groups, this automatically implies the inequality
%$$
%\dim_{\FF_3}\#\sel_{\phi}(E_{k})+
%\dim_{\FF_3}\#\sel_{\hat \phi}(E_{-27k})\geq 
%\rank E_k(\QQ),
%$$
recorded in  \cite[Prop.~42(i)]{BES}.)

\subsection{Binary cubic forms}

For any Dedekind domain $D$, we  write $V(D)$ for the set of binary cubic forms 
$f(u,v)=[a,3b,3c,d]$, with $a,b,c,d\in D$.  We put
$$
\disc(f)=-3b^2c^2+4ac^3+4b^3d+a^2d^2-6abcd,
$$
for the (reduced) discriminant of $f$. We put  $V(D)_A$ for the set of $f\in V(D)$ such that $\disc(f)=A$. 

Taking $D=\ZZ$, we shall adopt  the notation in \cite{BES} and 
introduce the sets
\begin{align*}
V(\ZZ)^{\text{sol}}&=\{f\in V(\ZZ): C_f(\QQ)\neq \emptyset\}\\
V(\ZZ)^{\text{loc sol}}&=\{f\in V(\ZZ): C_f(\QQ_p)\neq \emptyset~\forall \text{ primes $p$}\},
\end{align*}
where $C_f$ is given by \eqref{eq:Cf}.
Put 
$V(\ZZ)_A^{\text{sol}}$ and $V(\ZZ)_A^{\text{loc sol}}$ for the corresponding subsets restricted to have  discriminant $A$.
The group $\GL(\ZZ)$ acts naturally on $V(\RR)$ by linear change of variables and the discriminant is $\SL(\ZZ)$-invariant. We let $\cF$ be a fundamental domain for the action of $\SL(\ZZ)$ on 
 $V(\RR).$
We set
$$
\cF_X=\cF\cap \left\{f\in V(\RR): 0<|\disc(f)|\leq X\right\}.
$$
This set has finite volume. 
The following  result (cf.\ \cite[Thm.~37]{BES}) is due to Davenport \cite{dav} and  Davenport--Heilbronn \cite{DH}.

\begin{lemma}[Davenport--Heilbronn]\label{lem:Dav}
We have 
$$
\#\{f\in \cF_X \cap V(\ZZ): \text{$f$ irreducible}\} = c_1 X +o(X),
$$
as $X\to \infty$, where 
$c_1=\meas (\cF_1)=\frac{\pi^2}{3}$.
\end{lemma}

The following result allows us to associate 
a $\phi$-Selmer 
element of an elliptic curve $E_k$ over $\QQ$, with $k\in \ZZ$, 
to  an element of $V(\ZZ)$.

\begin{lemma}\label{lem:inj}
Let $k\in 3\ZZ$. There exists an injective map from the set of  non-identity elements of $\sel_{\hat \phi}(E_{-27k})$ to the set of irreducible $f\in 
\cF\cap V(\ZZ)_{4k}^{\text{loc sol}}$.
\end{lemma}

\begin{proof}
Assume throughout the proof that $k\in 3\ZZ$.
According to \cite[Thm.~33]{BES} there is a bijection between the 
$\SL(\QQ)$-equivalence classes of $V(\QQ)^{\text{loc sol}}_{4k}$ and elements of
$\sel_{\hat \phi}(E_{-27k})$.   Under this bijection, the identity element of 
$\sel_{\hat \phi}(E_{-27k})$ corresponds to the unique $\SL(\QQ)$-equivalence class of reducible forms in the space $V(\QQ)_{4k}$ (i.e.\ the orbit of $f(u,v)=ku^3+3uv^2$ under $\SL(\QQ)$).
Next, it follows from \cite[Thm.~34]{BES} that any 
$\SL(\QQ)$-equivalence class of $V(\QQ)^{\text{loc sol}}_{4k}$ contains an element of $V(\ZZ)_{4k}$.
This therefore produces the required map from 
 non-identity elements of $\sel_{\hat \phi}(E_{-27k})$ to $\SL(\ZZ)$-equivalence classes of irreducible elements of $V(\ZZ)_{4k}^{\text{loc sol}}$.  
Injectivity follows from injectivity of the map in \cite[Thm.~33]{BES}.
\end{proof}

\subsection{A positive proportion of curves with positive rank}

We now turn to the work of Kriz and Li \cite{kriz}. Let $\cD\subset \ZZ$ be the subset of  fundamental discriminants. Then $d\in \cD$ if and only if either $d\equiv 1 \bmod{4}$ is square-free, or $d=4d'$ for some square-free integer $d'\equiv \{2,3\}\bmod{4}$. We have 
the following result.

\begin{lemma}\label{lem:kriz}
We have  
$$
\#\{k\in \cD\cap (0,X]: \rank E_{-432k}(\QQ)=1\}\geq 
\frac{1}{2\pi^2}  X+o(X),
$$
as $X\to \infty$.
\end{lemma}
\begin{proof}
It follows from \cite[Thm.~1.8]{kriz}
that the analytic rank of $E_{-432k}$ is 1 for at least $\frac{1}{6}$ of fundamental discriminants $k\in \cD$.
By appealing to the celebrated work of  
Gross and Zagier \cite{GZ}, and Kolyvagin \cite{koly}, we deduce that the same is true for the arithmetic rank.
Finally, the statement of the lemma follows on recalling that
$
\#
 \cD\cap (0,X] = \frac{3}{\pi^2} X+o(X),
$
as $X\to \infty$. 
\end{proof}

\subsection{Proof of the main result}

Since $432k\in 3\ZZ$, 
we begin by deducing from Lemma \ref{lem:inj} that there is an injective map from soluble non-identity 
elements of  $\sel_{\hat \phi}(E_{27\cdot 432k})$ to  the set of irreducible elements in $ 
\cF\cap V(\ZZ)_{-4\cdot 432 k}^{\text{sol}}$.  
Recalling that $\hat\phi_{-432k}=\phi_{27\cdot 432k}$, we may  make the change of variables $27k \to -k$, in order to deduce that there is an injective map from  soluble non-identity 
elements of  $\sel_{\phi}(E_{-432k})$ to  the set of irreducible elements in $ 
\cF\cap V(\ZZ)_{64 k}^{\text{sol}}$.  
In view of \eqref{eq:home}, if  $k\in \ZZ$ is 
such that 
$\rank E_{-432k}(\QQ)=1$, then either 
$E_{27\cdot 432k}(\QQ)/\phi(E_{-432k}(\QQ))$ or 
$E_{-432k}(\QQ)/\hat \phi(E_{27\cdot 432k}(\QQ))$ 
must have size at least $3$. But then it follows that there are at least $2$  soluble non-identity elements belonging to 
$\sel_{\phi}(E_{-432k})$ or
$\sel_{\hat \phi}(E_{27\cdot 432k})$.
This implies that 
\begin{equation}\label{eq:ocado}
\begin{split}
\#\{f\in \cF_X\cap V(\ZZ)^{\text{sol}}: \text{$f$ irreducible}
 \} &\geq  
  \sum_{\substack{0<k\leq X/(4\cdot 432)\\ \rank E_{-432k}(\QQ)=1}} 2\\
&\geq \frac{1}{2^6\cdot 3^3\cdot \pi^2} X +o(X),
 \end{split}
 \end{equation} 
by  Lemma \ref{lem:kriz}.
We henceforth set
 $c_2=\frac{1}{2^6\cdot 3^3\cdot \pi^2}$ for the constant appearing in this bound.

We need a version of this which runs over irreducible elements of 
$V(\ZZ)^{\text{sol}}$ which are ordered by height, for which we ape an argument of Bhargava \cite[\S 4]{manjul}.
Let $B=[-1,1]^4$. We require a lower bound for the counting function
$$
N(X)
=
\#\{f\in   X^{\frac{1}{4}} B\cap V(\ZZ)^{\text{sol}}: \text{$f$ irreducible}\},
$$
as $X\to \infty$.
Since $\cF_X$ has finite volume, it follows that 
for any given $\ve>0$, there exists $r_\ve>0$ such that
$$
\meas(r_\ve B\cap \cF_X)\geq (1-\ve) \meas (\cF_X),
$$
for all $X>1$.  Thus  there exist constants $\delta>0$ and $r>0$ such that 
$$
\meas(r X^{\frac{1}{4}}B\cap \cF_X)\geq \left(1-\frac{c_2}{c_1}+\delta\right) \meas (\cF_X)=
\left(c_1-c_2+\delta c_1\right)X,
$$
for any $X>1$, where we recall from Lemma \ref{lem:Dav} that $c_1=\meas(\cF_1)=\frac{\pi^2}{3}$.
Now it easily follows from Lemma \ref{lem:Dav} and Hilbert's irreducibility theorem that 
$$
\#\{f\in  r X^{\frac{1}{4}} B\cap \cF_X \cap V(\ZZ): \text{$f$ 
irreducible}\}\geq   
(c_1-c_2+\delta c_1) X +o(X).
$$
The number of binary cubic forms contributing to the left hand side which don't belong $V(\ZZ)^{\text{sol}}$ is 
\begin{align*}
&\leq \#\{f\in  \cF_X \cap V(\ZZ): \text{$f$ 
irreducible}\}- 
\#\{f\in  \cF_X \cap V(\ZZ)^{\text{sol}}: \text{$f$ 
irreducible}\}
\\
&\leq   (c_1-c_2)X+o(X),
\end{align*}
by Lemma \ref{lem:Dav} and \eqref{eq:ocado}.
Putting $Y=r^{-4}X$, it therefore  follows that 
\begin{align*}
N(X)
&=
\#\{f\in  r Y^{\frac{1}{4}} B\cap V(\ZZ)^{\text{sol}}: \text{$f$ irreducible}\}\\
&\geq
\#\{f\in  r Y^{\frac{1}{4}} B\cap \cF_Y\cap V(\ZZ)^{\text{sol}}: \text{$f$ irreducible}\}\\
&\geq \delta c_1Y +o(Y).
\end{align*}
This completes the proof of Theorem \ref{t:main}.

%\begin{remark}
%Instead\footnote{Remove this remark! Too unimpressive!} of ordering the cubic surfaces \eqref{eq:cubics} by naive height, 
%we may construct an explicit compact region whose dilations produce surfaces with $\QQ$-rational points at least   0.00009\% of the time, using 
%a construction of Bhargava \cite[Rem.~18]{manjul}.
%Since $\cF_1$ has finite volume, there exists a compact set $B$ that is the closure of an open set containing the origin such that $\meas(B) = \meas(\cF_1)$ and $\meas(B \cap \cF_1) > (1 - \ve)\meas(\cF_1)$,
%for any $\ve > 0$. 
%Since $\ve$ can be made arbitrarily small we may adjust the above argument to conclude that 
%$$N(X)\geq (c_2-\tfrac{1}{1000})X+o(X),
%$$
%where $N(X)$ is given by \eqref{eq:NX}, but with our new compact set $B$.
%Since we have $(c_2-\tfrac{1}{1000})^2\geq ??$
%This therefore shows that at least ??\%  of the cubic surfaces \eqref{eq:cubics}
%contain a $\QQ$-rational point, when the binary cubic forms are constrained to lie in dilations of  $B$.
%\end{remark}

\end{document}